\documentclass[12pt,leqno]{article}
\usepackage{amssymb}
\usepackage[mathscr]{eucal}
\usepackage{amsmath,amssymb,latexsym,theorem,enumerate,ifthen}
\usepackage[english]{babel}
\usepackage{color,url}
\usepackage{appendix}
\usepackage{graphicx}
\usepackage{subcaption}

\newcommand{\NN}{\mathbb{N}}
\newcommand{\QQ}{\mathbb{Q}}
\newcommand{\RR}{\mathbb{R}}

\newcommand{\ZZ}{\mathbb{Z}}

\newcommand{\ba}{{\boldsymbol{a}}}

\newcommand{\bb}{{\boldsymbol{b}}}

\newcommand{\bk}{{\boldsymbol{k}}}

\newcommand{\bm}{{\boldsymbol{m}}}

\newcommand{\bx}{{\boldsymbol{x}}}

\newcommand{\by}{{\boldsymbol{y}}}

\newcommand{\bz}{{\boldsymbol{z}}}

\newcommand{\blambda}{{\boldsymbol{\lambda}}}

\newcommand{\bmu}{{\boldsymbol{\mu}}}

\newcommand{\bzero}{{\boldsymbol{0}}}

\newcommand{\cA}{{\mathcal A}}

\newcommand{\cJ}{\mathcal{J}}

\newcommand{\cX}{{\mathcal X}}

\newcommand{\PP}{\operatorname{\mathbb{P}}}

\newcommand{\sign}{\operatorname{sign}}

\newcommand{\comment}[1]{}

\renewcommand{\leq}{\leqslant}
\renewcommand{\geq}{\geqslant}

\newcommand{\proofend}{\hfill\mbox{$\Box$}}

\numberwithin{equation}{section}

\theoremstyle{change} \theorembodyfont{\em}
\newtheorem{Lem}{Lemma.}[section]
\newtheorem{Thm}[Lem]{Theorem.}
\newtheorem{Pro}[Lem]{Proposition.}

\newtheorem{Cor}[Lem]{Corollary.}

\theorembodyfont{\rm}
\newtheorem{Def}[Lem]{Definition.}
\newtheorem{Rem}[Lem]{Remark.}

\newtheorem{Ex}[Lem]{Example.}

\def\OnlyOnArXiv#1#2{\ifthenelse{\equal{#1}{Y}}{#2}{}}

\def\eq#1{{\rm(\ref{#1})}}
\long\def\Eq#1#2{\ifthenelse{\equal{#1}{*}}
  {\begin{equation*}\begin{aligned}#2\end{aligned}\end{equation*}}
  {\begin{equation}\begin{aligned}\label{#1}#2\end{aligned}\end{equation}}}

\newenvironment{proof}{\noindent{\bf Proof.}}{\proofend}

\begin{document}

\begin{center}
 {\bfseries\Large Basic properties of generalized $\psi$-estimators}

\vspace*{3mm}

{\sc\large
  M\'aty\'as $\text{Barczy}^{*}$,
  Zsolt $\text{P\'ales}^{**,\diamond}$ }

\end{center}

\vskip0.2cm

\noindent
 * HUN-REN–SZTE Analysis and Applications Research Group,
   Bolyai Institute, University of Szeged,
   Aradi v\'ertan\'uk tere 1, H--6720 Szeged, Hungary.

\noindent
 ** Institute of Mathematics, University of Debrecen,
    Pf.~400, H--4002 Debrecen, Hungary.

\noindent {E-mails:} barczy@math.u-szeged.hu (M. Barczy),
                  pales@science.unideb.hu  (Zs. P\'ales).

\noindent $\diamond$ Corresponding author.

\begin{center}
{\it Dedicated to the memory of Professor Zoltán Daróczy}
\end{center}

\vskip0.1cm


{\renewcommand{\thefootnote}{}
\footnote{\textit{2020 Mathematics Subject Classifications\/}: 62F10, 62D99, 26E60.}
\footnote{\textit{Key words and phrases\/}:
generalized $\psi$-estimator, $Z$-estimator, mean-type property, sensitivity, bisymmetry-type inequality, asymptotic properties.}
\vspace*{0.2cm}
\footnote{M\'aty\'as Barczy was supported by the project TKP2021-NVA-09.
Project no.\ TKP2021-NVA-09 has been implemented with the support
 provided by the Ministry of Culture and Innovation of Hungary from the National Research, Development and Innovation Fund,
 financed under the TKP2021-NVA funding scheme.
Zsolt P\'ales is supported by the K-134191 NKFIH Grant.}}

\vspace*{-10mm}

\begin{abstract}
We establish several properties of (weighted) generalized $\psi$-estimators introduced by Barczy and P\'ales in 2022: mean-type, monotonicity and sensitivity properties,
bisymmetry-type inequality and some asymptotic and continuity properties as well.
We also illustrate these properties by providing several examples including statistical ones as well.
\end{abstract}


\section{Introduction}
\label{section_intro}

In the theory of means several attempts were made in order to extend the notion of power and, more generally, of quasi-arithmetic means whose theory was developed already in the Thirties of the last century by Hardy, Littlewood and Pólya in their monograph \cite{HarLitPol34}.
In one direction, as an extension of power means, Gini \cite{Gin1938} introduced certain homogeneous means depending on two parameters (nowadays called Gini means).
In the other direction, in 1958, Bajraktarević \cite{Baj58} constructed a new class of means, generated by two functions, which generalized both the Gini means as well as the quasi-arithmetic means.
The next important development is due to Daróczy \cite{Dar71b,Dar72b} who introduced the notion of a deviation function and the concept of a deviation mean generated by a deviation function.
The second author in 1982 created further generalizations called quasi-deviation functions and quasi-deviation means \cite{Pal82a}.

In the theory of statistical estimators, Huber \cite{Hub64, Hub67} introduced a method to obtain estimation of an unknown parameter of the distribution of a random variable.
It was first observed by the authors in 2022 that the constructions of deviation means by Daróczy and $Z$-estimators (also called $\psi$-estimators) by Huber have many similarities, and therefore there should exist parallel results in both theories.
According to our knowledge, many results known for deviation means have not been elaborated for the setting of $Z$-estimators.
This paper attempts to establish new properties of $Z$-estimators that are motivated by analogous ones from the theory of deviation and quasi-deviation means.
In fact, we will derive new properties of (weighted) generalized $\psi$-estimators which were introduced by Barczy and P\'ales in \cite{BarPal2}  (see Definition \ref{Def_sign_change}).

Let us recall the notion of $Z$-estimators.
Let $(\xi_n)_{n\geq 1}$ be a sequence of independent and identically distributed (i.i.d.) random variables with values
 in a measurable space $(X,\cX)$ such that the distribution of $\xi_1$ depends on an unknown parameter $\vartheta$ belonging to a Borel measurable subset $\Theta\subseteq\RR$.
Based on the observations $\xi_1,\ldots,\xi_n$, the $Z$-estimator (where the letter Z refers to ''zero'') of the parameter $\vartheta$ is defined to be a solution $\widehat\vartheta_n:=\widehat\vartheta_n(\xi_1,\ldots,\xi_n)$ of the equation (with respect to the unknown parameter):
 \[
 \sum_{i=1}^n\psi(\xi_i,t)=0, \qquad t\in\Theta,
 \]
where $\psi:X\times\Theta\to\RR$ is a certain function which is measurable in its second variable with respect to the sigma-algebra $\cX$. In the statistical literature, the above $Z$-estimator is often called a $\psi$-estimator.
For a detailed exposition about $Z$-estimators, see, e.g., Kosorok \cite[Sections 2.2.5 and 13]{Kos} or van der Vaart \cite[Section 5]{Vaa}.

In our recent paper Barczy and P\'ales \cite{BarPal2}, we introduced the notion of weighted generalized $\psi$-estimators (recalled below in Definition \ref{Def_sign_change}),
 and we studied their existence and uniqueness.
In a companion paper Barczy and P\'ales \cite{BarPal3}, we solved the comparison problem and the equality problem for generalized $\psi$-estimators.

Throughout this paper, we fix the following notations.
Let $\NN$, $\ZZ_+$, $\QQ$, $\RR$, $\RR_+$ and $\RR_{++}$ denote the sets of positive integers,
 non-negative integers, rational numbers, real numbers, non-negative real numbers and positive real numbers, respectively.
A real interval will be called nondegenerate if it contains at least two distinct points.
For each $n\in\NN$, let us also introduce the set $\Lambda_n:=\RR_+^n\setminus\{(0,\ldots,0)\}$.

Throughout this paper, let $X$ be a nonempty set, $\Theta$ be a nondegenerate open interval of $\RR$.
Let $\Psi(X,\Theta)$ denote the class of real-valued functions $\psi:X\times\Theta\to\RR$ such that,
 for all $x\in X$, there exist $t_+,t_-\in\Theta$ such that $t_+<t_-$ and $\psi(x,t_+)>0>\psi(x,t_-)$.
Roughly speaking, a function $\psi\in\Psi(X,\Theta)$ satisfies the following property:
 for all $x\in X$, the function $t\ni\Theta\mapsto \psi(x,t)$ changes sign (from positive to negative) on the interval $\Theta$ at least once.

\begin{Def}\label{Def_sign_change}
Let $\Theta$ be a nondegenerate open interval of $\RR$. For a function $f:\Theta\to\RR$, consider the following three level sets
\[
  \Theta_{f>0}:=\{t\in \Theta: f(t)>0\},\qquad
  \Theta_{f=0}:=\{t\in \Theta: f(t)=0\},\qquad
  \Theta_{f<0}:=\{t\in \Theta: f(t)<0\}.
\]
We say that $\vartheta\in\Theta$ is a \emph{point of sign change (of decreasing type) for $f$} if
 \[
 f(t) > 0 \quad \text{for $t<\vartheta$,}
   \qquad \text{and} \qquad
    f(t)< 0 \quad  \text{for $t>\vartheta$.}
 \]
\end{Def}

Note that there can exist at most one element $\vartheta\in\Theta$ which is a point of sign change for $f$.
Further, if $f$ is continuous at a point $\vartheta$ of sign change, then $\vartheta$ is the unique zero of $f$.

\begin{Def}\label{Def_Tn}
We say that a function $\psi\in\Psi(X,\Theta)$
  \begin{enumerate}[(i)]
    \item \emph{possesses the property $[C]$ (briefly, $\psi$ is a $C$-function)} if
           it is continuous in its second variable, i.e., if, for all $x\in X$,
           the mapping $\Theta\ni t\mapsto \psi(x,t)$ is continuous.
    \item \emph{possesses the property $[T_n]$ (briefly, $\psi$ is a $T_n$-function)
           for some $n\in\NN$} if there exists a mapping $\vartheta_{n,\psi}:X^n\to\Theta$ such that,
           for all $\pmb{x}=(x_1,\dots,x_n)\in X^n$ and $t\in\Theta$,
           \begin{align*}
             \psi_{\pmb{x}}(t):=\sum_{i=1}^n \psi(x_i,t) \begin{cases}
                 > 0 & \text{if $t<\vartheta_{n,\psi}(\pmb{x})$,}\\
                 < 0 & \text{if $t>\vartheta_{n,\psi}(\pmb{x})$},
            \end{cases}
           \end{align*}
          that is, for all $\pmb{x}\in X^n$, the value $\vartheta_{n,\psi}(\pmb{x})$ is a point of sign change for the function $\psi_{\pmb{x}}$. If there is no confusion, instead of $\vartheta_{n,\psi}$ we simply write $\vartheta_n$.
          We may call $\vartheta_{n,\psi}(\pmb{x})$ as a generalized $\psi$-estimator for
         some unknown parameter in $\Theta$ based on the realization $\bx=(x_1,\ldots,x_n)\in X^n$. If, for each $n\in\NN$, $\psi$ is a $T_n$-function, then we say that \emph{$\psi$ possesses the property $[T]$ (briefly, $\psi$ is a $T$-function)}.
    \item \emph{possesses the property $[Z_n]$ (briefly, $\psi$ is a $Z_n$-function) for some $n\in\NN$} if it is a $T_n$-function and
    \[
   \psi_{\pmb{x}}(\vartheta_{n,\psi}(\pmb{x}))=\sum_{i=1}^n \psi(x_i,\vartheta_{n,\psi}(\pmb{x}))= 0
    \qquad \text{for all}\quad \pmb{x}=(x_1,\ldots,x_n)\in X^n.
    \]
    If, for each $n\in\NN$, $\psi$ is a $Z_n$-function, then we say that \emph{$\psi$ possesses the property $[Z]$ (briefly, $\psi$ is a $Z$-function)}.
    \item \emph{possesses the property $[T_n^{\pmb{\lambda}}]$ for some $n\in\NN$ and $\pmb{\lambda}=(\lambda_1,\ldots,\lambda_n)\in\Lambda_n$ (briefly, $\psi$ is a $T_n^{\pmb{\lambda}}$-function)} if there exists a mapping $\vartheta_{n,\psi}^{\pmb{\lambda}}:X^n\to\Theta$ such that, for all $\pmb{x}=(x_1,\dots,x_n)\in X^n$ and $t\in\Theta$,
          \begin{align*}
           \psi_{\pmb{x},\pmb{\lambda}}(t):= \sum_{i=1}^n \lambda_i\psi(x_i,t) \begin{cases}
                 > 0 & \text{if $t<\vartheta_{n,\psi}^{\pmb{\lambda}}(\pmb{x})$,}\\
                 < 0 & \text{if $t>\vartheta_{n,\psi}^{\pmb{\lambda}}(\pmb{x})$},
             \end{cases}
           \end{align*}
           that is, for all $\pmb{x}\in X^n$, the value $\vartheta_{n,\psi}^{\pmb{\lambda}}(\pmb{x})$ is
           a point of sign change for the function $\psi_{\pmb{x},\pmb{\lambda}}$.
           If there is no confusion, instead of $\vartheta_{n,\psi}^{\pmb{\lambda}}$ we simply write $\vartheta_n^{\pmb{\lambda}}$.
          We may call $\vartheta_{n,\psi}^{\pmb{\lambda}}(\pmb{x})$
          as a weighted generalized $\psi$-estimator for some unknown parameter in $\Theta$ based
          on the realization $\bx=(x_1,\ldots,x_n)\in X^n$ and weights $(\lambda_1,\ldots,\lambda_n)\in\Lambda_n$.
    \item \emph{possesses the property $[Z_n^{\pmb{\lambda}}]$ for some $n\in\NN$ and $\pmb{\lambda}=(\lambda_1,\ldots,\lambda_n)\in\Lambda_n$ (briefly, $\psi$ is a $Z_n^{\pmb{\lambda}}$-function)} if it is a $T_n^{\pmb{\lambda}}$-function and
    \[
        \psi_{\pmb{x,\blambda}}(\vartheta_{n,\psi}^\blambda(\pmb{x}))
           =\sum_{i=1}^n \lambda_i \psi(x_i,\vartheta_{n,\psi}^\blambda(\bx))= 0
    \qquad \text{for all}\quad \pmb{x}=(x_1,\ldots,x_n)\in X^n.
    \]
      \item \emph{possesses the property $[W_n]$ for some $n\in\NN$ (briefly, $\psi$ is a $W_n$-function)}
           if it is a $T_n^{\pmb{\lambda}}$-function for all $\pmb{\lambda}\in\Lambda_n$. If, for each $n\in\NN$, $\psi$ is a $W_n$-function, then we say that \emph{$\psi$ possesses the property $[W]$ (briefly, $\psi$ is a $W$-function)}.
   \end{enumerate}
\end{Def}

It can be seen that if $\psi$ is continuous in its second variable, and, for some $n\in\NN$, it is a $T_n$-function, then it also a $Z_n$-function.
Further, if $\psi\in\Psi(X,\Theta)$ is a $T_n$-function for some $n\in\NN$, then $\vartheta_{n,\psi}$ is symmetric in the sense that
$\vartheta_{n,\psi}(x_1,\ldots,x_n) = \vartheta_{n,\psi}(x_{\pi(1)},\ldots,x_{\pi(n)})$ holds for all $x_1,\ldots,x_n\in X$ and all permutations $(\pi(1),\ldots,\pi(n))$ of $(1,\ldots,n)$.

Assume that $\psi$ possesses property $[W]$. Then it also possesses property $[T]$, and, for all $n\in\NN$, $\vartheta_{n,\psi}=\vartheta_{n,\psi}^{(1,\dots,1)}$ holds on $X^n$.
More generally, for all $k,n_1,\dots,n_k\in\NN$ and $x_1,\dots,x_k\in X$, we have
\Eq{*}{
  \vartheta_{n_1+\dots+n_k,\psi}(\underbrace{x_1,\dots,x_1}_{n_1},\dots,\underbrace{x_k,\dots,x_k}_{n_k})
  =\vartheta_{k,\psi}^{(n_1,\dots,n_k)}(x_1,\dots,x_k).
}
In addition, $\vartheta_{n,\psi}^{\blambda}=\vartheta_{n,\psi}^{s\blambda}$ holds on $X^n$ for all $n\in\NN$, $s>0$ and $\blambda\in\Lambda_n$.
We refer to this last property that the generalized weighted $\psi$-estimator is null-homogeneous in the weights.

Given $q\in\NN$ and properties $[P_1], \ldots, [P_q]$ which can be any of the properties introduced in Definition~\ref{Def_Tn}, the subclass of $\Psi(X,\Theta)$ consisting of elements possessing the properties $[P_1],\ldots,[P_q]$, will be denoted by $\Psi[P_1,\ldots,P_q](X,\Theta)$.

The paper is organized as follows.
In Section \ref{section_ine_mon_prop}, we establish a mean-type property of generalized $\psi$-estimators (see Theorem \ref{Thm_psi_becsles_mean_prop}),
 and we prove monotonicity of weighted generalized $\psi$-estimators in the weights belonging to a line (see Theorem \ref{Thm_theta_monotone}).
We also establish a bisymmetry-type inequality for weighted generalized $\psi$-estimators (see Theorem \ref{Thm_bisymmetry}),
 which is a generalization of bisymmetry inequality for weighted quasi-deviation means
 due to P\'ales \cite[Theorem 3.2]{Pal82a}.

Section \ref{section_asy_con_prop} is devoted to establishing
asymptotic and regulariy properties of generalized $\psi$-estimators.
Theorem \ref{Thm_psi_est_eq_3_T} provides a continuity property,
 and we also formulate a consequence of this result (see Corollary \ref{Cor_psi_est_eq_3_T}),
 which points out an interesting feature of generalized $\psi$-estimators.
Namely, given $n,m,\ell\in\NN$, suppose that one observes $\bz\in X^m$ and then $\by\in X^n$ repeatedly $\ell$-times.
In this case it turns out that, roughly speaking, the generalized $\psi$-estimator does not essentially depend
 on $\bz\in X^m$ in the sense that its effect disappears as $\ell\to\infty$.
One could say that, in this scenario, the observation $\bz$ can be considered asymptotically as an outlier.
In Definition \ref{Def_sensitivity}, we introduce the notion of sensitivity for estimators,
 and then we prove that this property holds for generalized $\psi$-estimators
(see Theorem \ref{Thm_psi_est_infinit} and Corollary \ref{Cor_infinit}).

In Section \ref{section_examples}, we give several applications of
 our results in Sections \ref{section_ine_mon_prop} and \ref{section_asy_con_prop}.
It turns out that Theorems \ref{Thm_psi_becsles_mean_prop} and \ref{Thm_bisymmetry2}
 can efficiently be used to show that a certain estimator cannot be
 a generalized $\psi$-estimator, see Examples \ref{Ex2} and \ref{Ex4}.
We also present some examples for sensitive and non-sensitive estimators,
 which have important role in statistics (see Example \ref{Ex5}).
Further, we calculate the weighted generalized $\psi$-estimators for two particular
 functions $\psi$ corresponding to MLE estimators,
 and then we formulate the 2-variable bisymmetry-type inequality in these special cases (see Example \ref{Ex1}).

We close the paper with Appendix \ref{App1}, where we recall a result on the characterisation of quasi-affine functions, which is a folklore in the literature.
Namely, given a nondegenerate interval $I\subseteq \RR$ and $f:I\to\RR$, we prove that
 $f$ is quasi-affine if and only if $f$ is monotone.
For completeness, we provide a proof as well, since we could not address any reference for it.
This result is used in the proof of Theorem \ref{Thm_theta_monotone}.

\section{Inequality and monotonicity properties of generalized $\psi$-estimators}
\label{section_ine_mon_prop}

First, we establish a mean-type property of generalized $\psi$-estimators, which is a generalization of Proposition 1.4 in Barczy and P\'ales \cite{BarPal3}.
In the theory of means, the corresponding property is called internity, see, e.g., P\'ales \cite[page 65]{Pal1984}.
In the sequel, we will adopt the following convention: if $k,n_1,\dots,n_k\in\NN$ and
 $\bx_1=(x_{1,1},\dots,x_{1,n_1})\in X^{n_1}$, \dots, $\bx_k=(x_{k,1},\dots,x_{k,n_k})\in X^{n_k}$,
 then their \emph{concatenation} $(\bx_1,\dots,\bx_k)\in X^{n_1}\times\cdots\times X^{n_k}$ is identified by
\[
 (x_{1,1},\dots,x_{1,n_1}, \dots,
 x_{k,1},\dots,x_{k,n_k})\in X^{n_1+\dots+n_k}.
\]

\begin{Thm}\label{Thm_psi_becsles_mean_prop}
Let $\psi\in\Psi[T](X,\Theta)$.
Then, for each $k,n_1,\dots,n_k\in\NN$ and $\bx_1\in X^{n_1}$, \dots, $\bx_k\in X^{n_k}$, we have
\begin{align}\label{help_mean_prop_csop}
\begin{split}
 \min(\vartheta_{n_1}(\bx_1),\dots,\vartheta_{n_k}(\bx_k))
  \leq \vartheta_{n_1+\dots+n_k}(\bx_1,\dots,\bx_k)\leq \max(\vartheta_{n_1}(\bx_1),\dots,\vartheta_{n_k}(\bx_k)).
\end{split}
\end{align}
Furthermore, if $\psi\in\Psi[Z](X,\Theta)$ and not all of the values $\vartheta_{n_1}(\bx_1),\dots,\vartheta_{n_k}(\bx_k)$ are equal, then both inequalities in \eqref{help_mean_prop_csop} are strict.
\end{Thm}

\begin{proof}
Let $\bx_1=(x_{1,1},\dots,x_{1,n_1})\in X^{n_1}$, \dots, $\bx_k=(x_{k,1},\dots,x_{k,n_k})\in X^{n_k}$, and let $t\in\Theta$ with $t<\min(\vartheta_{n_1}(\bx_1),\dots,\vartheta_{n_k}(\bx_k))$ be arbitrary.
Then, for all $i\in\{1,\dots,k\}$, we have $t<\vartheta_{n_i}(\bx_i)$, and hence
\begin{align*}
   \sum_{j=1}^{n_i} \psi(x_{i,j},t)>0.
\end{align*}
Summing up these inequalities for $i\in\{1,\dots,k\}$ side by side, we get
\begin{align*}
   \sum_{i=1}^k\sum_{j=1}^{n_i} \psi(x_{i,j},t)>0,
\end{align*}
which yields that $t\leq \vartheta_{n_1+\dots+n_k}(\bx_1,\dots,\bx_k)$.

Upon taking the limit $t\uparrow \min(\vartheta_{n_1}(\bx_1),\dots,\vartheta_{n_k}(\bx_k))$,
 we can conclude that the left hand side inequality in \eqref{help_mean_prop_csop} is valid.
A similar argument shows that the right hand side inequality in \eqref{help_mean_prop_csop} is also true.

To complete the proof, assume that $\psi\in\Psi[Z](X,\Theta)$ and that not all of the values $\vartheta_{n_1}(\bx_1)$, \dots, $\vartheta_{n_k}(\bx_k)$ are equal.
Then there exist $i_1,i_2\in\{1,\ldots,k\}$ with $i_1\ne i_2$ and
 \[
    \vartheta_{n_{i_1}}(\bx_{i_1}) = \min(\vartheta_{n_1}(\bx_1),\dots,\vartheta_{n_k}(\bx_k))
      < \max(\vartheta_{n_1}(\bx_1),\dots,\vartheta_{n_k}(\bx_k))=\vartheta_{n_{i_2}}(\bx_{i_2}).
 \]
Since $\psi$ is a $Z_{n_i}$-function, $i\in\{1,\ldots,k\}$, we have that
 \begin{align*}
   &\sum_{j=1}^{n_{i_1}} \psi(x_{i_1,j}, \vartheta_{n_{i_1}}(\bx_{i_1}))=0,\qquad \sum_{j=1}^{n_{i_2}} \psi(x_{i_2,j}, \vartheta_{n_{i_1}}(\bx_{i_1}))>0,\\
   &\sum_{j=1}^{n_i} \psi(x_{i,j}, \vartheta_{n_{i_1}}(\bx_{i_1}))\geq 0, \qquad i\in\{1,\ldots,k\}\setminus\{i_1,i_2\},
 \end{align*}
 and
 \begin{align*}
   &\sum_{j=1}^{n_{i_2}} \psi(x_{i_2,j}, \vartheta_{n_{i_2}}(\bx_{i_2}))=0,\qquad
    \sum_{j=1}^{n_{i_1}} \psi(x_{i_1,j}, \vartheta_{n_{i_2}}(\bx_{i_2}))<0,\\
   &\sum_{j=1}^{n_i} \psi(x_{i,j}, \vartheta_{n_{i_2}}(\bx_{i_2}))\leq 0, \qquad i\in\{1,\ldots,k\}\setminus\{i_1,i_2\}.
 \end{align*}
These inequalities imply that
\begin{align*}
 \sum_{i=1}^k  \sum_{j=1}^{n_i} \psi(x_{i,j}, \vartheta_{n_{i_1}}(\bx_{i_1})) >0
 \qquad\mbox{and}\qquad
 \sum_{i=1}^k  \sum_{j=1}^{n_i} \psi(x_{i,j}, \vartheta_{n_{i_2}}(\bx_{i_2})) < 0.
\end{align*}
Thus, using that $\psi$ is also a $Z_{n_1+\cdots+n_k}$-function, it follows that
 \[
 \vartheta_{n_{i_1}}(\bx_{i_1}) < \vartheta_{n_1+\dots+n_k}(\bx_1,\dots,\bx_k) < \vartheta_{n_{i_2}}(\bx_{i_2}),
 \]
 as desired.
\end{proof}

We note that the particular case of Theorem \ref{Thm_psi_becsles_mean_prop} when $n_1:=1,\ldots,n_k:=1$ was established in Proposition 1.4 in Barczy and P\'ales \cite{BarPal3}.
Furthermore, note that, given $k,n_1,\dots,n_k\in\NN$ and $\bx_1\in X^{n_1}$, \dots, $\bx_k\in X^{n_k}$, the inequality \eqref{help_mean_prop_csop} remains valid if, instead of $\psi\in\Psi[T](X,\Theta)$, we only suppose that $\psi$ has the properties
$[T_{n_i}]$, $i=1,\ldots,k$, and $[T_{n_1+\cdots+n_k}]$.

We recall that, given a non-empty convex subset $S$ of $\RR^n$, a function $f:S\to\RR$ is called quasi-concave if for all $x,y\in S$ and $t\in(0,1)$, the inequality
 \begin{align*}
   f(tx + (1-t)y) \geq \min\{f(x),f(y)\}
 \end{align*}
 holds.
A function $f:S\to\RR$ is called quasi-convex if for all $x,y\in S$
 and $t\in(0,1)$, the inequality
 \begin{align*}
   f(tx + (1-t)y) \leq \max\{f(x),f(y)\}
 \end{align*}
 holds.
Finally, a function $f:S\to\RR$ is said to be quasi-affine if it is both quasi-convex and quasi-concave. It is known that $f:S\to\RR$ is quasi-convex
 if and only if, for all $\alpha\in\RR$, the level set $\{x\in S : f(x)\leq \alpha \}$ is convex.
Further, if the set $\{x\in S : f(x) = \alpha \}$ is convex for all $\alpha\in\RR$, and $f$ is continuous, then $f$ is quasi-affine,
 see, e.g., Greenberg and Pierskalla \cite{GrePie}.
It is known that, given a nondegenerate interval $I\subseteq \RR$,
a function $f:I\to\RR$ is quasi-affine if and only if it is monotone on $I$, see Proposition \ref{Pro_quasi}.
Finally, we mention that in the literature quasi-affine functions sometimes called quasi-linear functions as well.

\begin{Thm}\label{Thm_theta_monotone}
Let $n\in\NN$ and let $\psi\in\Psi[W_n](X,\Theta)$.
For $\ba,\bb\in\RR^n$, denote $\cJ_{\ba,\bb}:=\{s\in\RR : s\ba +\bb \in\Lambda_n\}$.
Then, for all $\bx\in X^n$, the function
 \[
   \cJ_{\ba,\bb}\ni s\mapsto \vartheta_n^{s\ba+\bb}(\bx)\in\Theta
 \]
 is monotone.
\end{Thm}

\begin{proof}
Let $n\in\NN$, $\bx:=(x_1,\ldots,x_n)\in X^n$, $\ba:=(a_1,\ldots,a_n)\in\RR^n$, and $\bb:=(b_1,\ldots,b_n)\in\RR^n$ be fixed arbitrarily.
Using that $\Lambda_n$ is convex, it follows that $\cJ_{\ba,\bb}$ is a convex subset of $\RR$, and then it is an interval.
Note that it can happen that $\cJ_{\ba,\bb}$ is empty (e.g., in case of $\ba=\bzero$ and $\bb=(-1,\ldots,-1)^\top\in\RR^n$), when there is nothing to prove.

In what follows, let us assume that $\cJ_{\ba,\bb}$ is non-empty. Define the function $g:\cJ_{\ba,\bb}\to\Theta$ by $g(s):=\vartheta_n^{s\ba+\bb}(\bx)$, $s\in\cJ_{\ba,\bb}$.
Since $\psi$ is a $T_n^{s\ba+\bb}$-function for all $s\in\cJ_{\ba,\bb}$, we have that $g$ is well-defined.
In view of Proposition \ref{Pro_quasi}, it is sufficient to show that $g$ is quasi-affine on $\cJ_{\ba,\bb}$,
i.e., it is quasi-convex and quasi-concave.
To see that $g$ is quasi-convex, let $u,v\in\cJ_{\ba,\bb}$ and $t\in(0,1)$. Let $c\in\Theta$ with $c>\max(g(u),g(v))$ be arbitrary (since $\Theta$ is open, such a $c$ exists).
Since $\psi$ has the properties $[T_n^{u\ba+\bb}]$ and $[T_n^{v\ba+\bb}]$, we get
 \begin{align*}
   \sum_{i=1}^n (ua_i + b_i) \psi(x_i,c)<0
   \qquad\mbox{and}\qquad
   \sum_{i=1}^n (va_i + b_i) \psi(x_i,c)<0.
 \end{align*}
Multiplying these inequalities by $t$ and by $1-t$, respectively,
and adding the inequalities so obtained, we arrive at
 \begin{align*}
   \sum_{i=1}^n ((tu+(1-t)v)a_i + b_i) \psi(x_i,c)<0.
 \end{align*}
This implies that $g(tu+(1-t)v)\leq c$.
Upon taking the limit $c\downarrow\max(g(u),g(v))$, it follows that $g(tu+(1-t)v)\leq\max(g(u),g(v))$, which shows the quasi-convexity of $g$.

The verification of the quasi-concavity of $g$ is analogous, therefore, it is omitted.
\end{proof}

\begin{Rem}
We point out that Theorem~\ref{Thm_theta_monotone} implies the inequality \eqref{help_mean_prop_csop} in
Theorem~\ref{Thm_psi_becsles_mean_prop} in the particular case $k=2$, provided that $\psi$ possesses the property $[W]$. To see this, let $n,m\in\NN$, $\by=(y_1,\ldots,y_n)\in X^n$, $\bz=(z_1,\ldots,z_m)\in X^m$, and let $\psi\in\Psi[W](X,\Theta)$.
With the notations $\bx:=(\by,\bz)\in X^n\times X^m$,
 \Eq{*}{
  \ba:=((-1,\dots,-1),(1,\dots,1))\in\ZZ^n\times\ZZ^m,\qquad\mbox{and}\qquad \bb:=((1,\dots,1),(0,\dots,0))\in\ZZ^n\times\ZZ^m,
 }
we get that $\bb\in\Lambda_{n+m}$ and $\ba+\bb=((0,\dots,0),(1,\dots,1))\in\Lambda_{n+m}$, and hence $0,1\in\cJ_{\ba,\bb}$.
Therefore, by the convexity of $\cJ_{\ba,\bb}$, we have that $[0,1]\subseteq \cJ_{\ba,\bb}$.
Consequently, Theorem \ref{Thm_theta_monotone} implies that
 the function $f\colon [0,1]\to\Theta$, $f(s):=\vartheta_{n+m}^{s\ba+\bb}(\by,\bz)$, $s\in[0,1]$, is monotone.
Thus, it follows that
\Eq{help_1}{
 \min(f(0),f(1))
  \leq f\big(\tfrac{1}{2}\big)
  \leq\max(f(0),f(1)).
}
Here $f(0)=\vartheta_{n+m}^{\bb}(\by,\bz)=\vartheta_{n}(\by)$, $f(1)=\vartheta_{n+m}^{\ba+\bb}(\by,\bz)=\vartheta_{m}(\bz)$ and
 \Eq{*}{
 f\big(\tfrac{1}{2}\big)
 =\vartheta_{n+m}^{\frac{1}{2}\ba+\bb}(\by,\bz)
 =\vartheta_{n+m}^{\ba+2\bb}(\by,\bz)
 =\vartheta_{n+m}(\by,\bz),
 }
 since $\ba+2\bb=\big((1,\dots,1),(1,\dots,1)\big)\in\NN^n\times\NN^m$
and we used that the generalized weighted $\psi$-estimator is null-homogeneous in the weights
 (see the discussion after Definition \ref{Def_Tn}).
Thus, the inequality \eqref{help_1} can be rewritten as
\Eq{*}{
 \min(\vartheta_{n}(\by),\vartheta_{m}(\bz))
  \leq \vartheta_{n+m}(\by,\bz)
  \leq\max(\vartheta_{n}(\by),\vartheta_{m}(\bz)),
}
which is the particular case of inequality \eqref{help_mean_prop_csop} with $k:=2$, $n_1:=n$, $n_2:=m$, $\bx_1:=\by$, and $\bx_2:=\bz$. It can also be shown that the inequality \eqref{help_mean_prop_csop} follows for general $k\in\NN$ from the particular case $k=2$ with induction with respect to $k$.
\proofend
\end{Rem}

Next, we prove a bisymmetry-type inequality for weighted generalized $\psi$-estimators.
Our result can be considered as a generalization of the bisymmetry inequality for weighted quasi-deviation means due to P\'ales \cite[Theorem 3.2]{Pal82a}.

\begin{Thm}[Bisymmetry-type inequality for weighted generalized $\psi$-estimators]\label{Thm_bisymmetry}
Let $n,m\in\NN$ and let $\psi\in\Psi[W_n,W_m](X,\Theta)$.
Then, for all $x_{i,j}\in X$, $\lambda_{i,j}\in\RR_+$, $i\in\{1,\dots,n\}$, $j\in\{1,\dots,m\}$ with
\Eq{lambda-positivity}{
 \sum_{i=1}^n \lambda_{i,j}>0,
 \quad j\in\{1,\dots,m\} \qquad\mbox{and}\qquad
 \sum_{j=1}^m \lambda_{i,j}>0,
 \quad i\in\{1,\dots,n\},
}
we have
\Eq{*}{
  \min_{i\in\{1,\dots,n\}}
  \vartheta_m^{(\lambda_{i,1},\dots,\lambda_{i,m})}(x_{i,1},\dots,x_{i,m})
  \leq \max_{j\in\{1,\dots,m\}}
  \vartheta_n^{(\lambda_{1,j},\dots,\lambda_{n,j})}(x_{1,j},\dots,x_{n,j}).
}
\end{Thm}

\begin{proof} To the contrary of the assertion, assume that
\Eq{*}{
  \min_{i\in\{1,\dots,n\}}
  \vartheta_m^{(\lambda_{i,1},\dots,\lambda_{i,m})}(x_{i,1},\dots,x_{i,m})
  > \max_{j\in\{1,\dots,m\}}
  \vartheta_n^{(\lambda_{1,j},\dots,\lambda_{n,j})}(x_{1,j},\dots,x_{n,j}).
}
holds for some $x_{i,j}\in X$, $\lambda_{i,j}\in\RR_+$, $i\in\{1,\dots,n\}$, $j\in\{1,\dots,m\}$ satisfying \eq{lambda-positivity}. Then there exists $t\in\Theta$ such that, for all $i\in\{1,\dots,n\}$ and $j\in\{1,\dots,m\}$,
\Eq{*}{
   \vartheta_m^{(\lambda_{i,1},\dots,\lambda_{i,m})}(x_{i,1},\dots,x_{i,m})
   >t>\vartheta_n^{(\lambda_{1,j},\dots,\lambda_{n,j})}(x_{1,j},\dots,x_{n,j}).
}
By our assumption, $\psi$ is a $T_n^{\blambda}$- and a $T_m^{\bmu}$-function for all $\blambda\in\Lambda_n$ and $\bmu\in\Lambda_m$, respectively. Therefore, for all $i\in\{1,\dots,n\}$ and $j\in\{1,\dots,m\}$, we get that
\Eq{*}{
  \sum_{\beta=1}^m \lambda_{i,\beta}\psi(x_{i,\beta},t)>0\qquad\mbox{and}\qquad
  0>\sum_{\alpha=1}^n \lambda_{\alpha,j}\psi(x_{\alpha,j},t).
}
Summing up the first and second inequalities side by side with respect to $i\in\{1,\dots,n\}$ and  $j\in\{1,\dots,m\}$, respectively, we can conclude that
\Eq{*}{
  \sum_{i=1}^n\sum_{\beta=1}^m \lambda_{i,\beta}\psi(x_{i,\beta},t)>0 \qquad\mbox{and}\qquad
  0>\sum_{j=1}^m\sum_{\alpha=1}^n \lambda_{\alpha,j}\psi(x_{\alpha,j},t),
}
which leads us to an obvious contradiction.
\end{proof}

With the same argument, taking $\lambda_{i,j}:=1$ for $i\in\{1,\dots,n\}$ and  $j\in\{1,\dots,m\}$,
 one can also verify an analogous result for non-weighted generalized $\psi$-estimators.

\begin{Thm}[Bisymmetry-type inequality for generalized $\psi$-estimators]\label{Thm_bisymmetry2}
Let $n,m\in\NN$, and let $\psi\in\Psi[T_n,T_m](X,\Theta)$. Then, for all $x_{i,j}\in X$, $i\in\{1,\dots,n\}$, $j\in\{1,\dots,m\}$,
\Eq{*}{
  \min_{i\in\{1,\dots,n\}}
  \vartheta_m(x_{i,1},\dots,x_{i,m})
  \leq \max_{j\in\{1,\dots,m\}}
  \vartheta_n(x_{1,j},\dots,x_{n,j}).
}
\end{Thm}

\begin{Cor} {\bf (2-variable bisymmetry-type inequality for weighted generalized $\psi$-esti\-mators)}
Let $\psi\in\Psi[W_2](X,\Theta)$. Then we have
 \begin{align}\label{ineq_22_bisymmeytry}
  \min( \vartheta_{2}^{(\alpha,\beta)}(x,y), \vartheta_{2}^{(\gamma,\delta)}(u,v))
      \leq \max(\vartheta_{2}^{(\alpha,\gamma)}(x,u), \vartheta_{2}^{(\beta,\delta)}(y,v))
 \end{align}
 for all $x,y,u,v\in X$, $\alpha,\beta,\gamma,\delta\in\RR_+$ such that $\alpha+\beta>0$,
 $\gamma+\delta> 0$, \ $\alpha+\gamma>0$ and $\beta+\delta> 0$.
\end{Cor}

\begin{proof} Let $x,y,u,v\in X$ and $\alpha,\beta,\gamma,\delta\in\RR_+$ with $\alpha+\beta>0$,
 $\gamma+\delta>0$, $\alpha+\gamma>0$ and $\beta+\delta>0$ and apply Theorem \ref{Thm_bisymmetry}
 for $n:=m:=2$ with $x_{1,1}:=x$, $x_{1,2}:=y$, $x_{2,1}:=u$, $x_{2,2}:=v$, and
 $\lambda_{1,1}:=\alpha$, $\lambda_{1,2}:=\beta$, $\lambda_{2,1}:=\gamma$, $\lambda_{2,2}:=\delta$.
\end{proof}

\section{Asymptotic and regularity properties of generalized $\psi$-estimators}
\label{section_asy_con_prop}

Our first statement is a generalization of Lemma 1.7 in Barczy and P\'ales \cite{BarPal3}.
For its formulation, we introduce two further notations.
Firstly, for each $k,n\in\NN$ and $\bx\in X^n$, let $k\odot \bx$ denote the $(k\cdot n)$-tuple $(\bx,\ldots,\bx)\in (X^n)^k$.
Secondly, for each $n\in\NN$, $\bk=(k_1,\dots,k_n)\in\ZZ_+^n$ with $k_1+\dots+k_n>0$ and  $\bx=(x_1,\dots,x_n)\in X^{n}$, let $\bk\otimes\bx$ be defined by
\[
  \bk\otimes\bx
  =(k_1\odot x_1,\dots,k_n\odot x_n)
  \in X^{k_1+\dots+k_n}.
\]
Here we adopt the convention that if, for some $i\in\{1,\dots,n\}$, $k_i=0$, then $k_i\odot x_i$ is cancelled in the above formula.
Note that if $n\in\NN$, $\bx=(x_1,\dots,x_n)\in X^n$, $\bk:=(k,\dots,k)\in\NN^n$ with some $k\in\NN$,
 and $\psi$ has the property $[T_{nk}]$, then, by symmetry (see the paragraph after Definition \ref{Def_Tn}), we can easily see that
 \begin{align*}
    \vartheta_{nk}(\bk\otimes\bx) = \vartheta_{nk}(k\odot \bx).
 \end{align*}

\begin{Thm}\label{Thm_psi_est_eq_3_T}
Let $\psi\in\Psi[T](X,\Theta)$, $n\in\NN$ and let $\bk=(k_1,\dots,k_n),\bm=(m_1,\dots,m_n)\in\ZZ_+^n$ with $k_1+\dots+k_n>0$. Then, for all $\bx\in X^n$,
 we have
 \[
  \lim_{\ell\to\infty}\vartheta_{\ell(k_1+\dots+k_n)+m_1+\dots+m_n}((\ell\bk+\bm)\otimes\bx)
   =\vartheta_{k_1+\dots+k_n}(\bk\otimes\bx),
 \]
 where $\ell\bk := (\ell k_1,\ldots,\ell k_n)\in\ZZ_+^n$, $\ell\in\NN$.
\end{Thm}

\begin{proof} Let $\bx=(x_1,\dots,x_n)\in X^n$ be fixed arbitrarily.
For each $\ell\in\NN$, denote
 \[
   t_\ell:=\vartheta_{\ell(k_1+\dots+k_n)+m_1+\dots+m_n}((\ell\bk+\bm)\otimes\bx).
 \]
We need to show that $t_\ell$ converges to $\vartheta_{k_1+\dots+k_n}(\bk\otimes\bx)$ as $\ell \to\infty$.

Let $t',t''\in\Theta$ be arbitrary such that $t'<\vartheta_{k_1+\dots+k_n}(\bk\otimes\bx)<t''$ (since $\Theta$ is open such values $t'$ and $t''$ exist). Then
 \[
  \sum_{i=1}^n(\ell k_i+m_i)\psi(x_i,t')
  =\ell\bigg(\sum_{i=1}^n k_i\psi(x_i,t')+\frac1\ell\sum_{i=1}^n m_i\psi(x_i,t')\bigg)>0
 \]
 if $\ell$ is large enough, because $t'<\vartheta_{k_1+\dots+k_n}(\bk\otimes\bx)$ and
 \[
   \lim_{\ell\to\infty}\bigg(\sum_{i=1}^n k_i\psi(x_i,t')+\frac1\ell\sum_{i=1}^n m_i\psi(x_i,t')\bigg)=\sum_{i=1}^n k_i\psi(x_i,t')>0.
 \]
Similarly, we get that the inequality
 \[
\sum_{i=1}^n(\ell k_i+m_i)\psi(x_i,t'')<0
 \]
 is valid if $\ell$ is large enough.
Therefore, the point of sign change $t_\ell$ of the function
 \[
  \Theta\ni t\mapsto \sum_{i=1}^n(\ell k_i+m_i)\psi(x_i,t)
 \]
is in the interval $[t',t'']$ for $\ell$ large enough, that is, there exists $\ell_0\in\NN$ such that $t_\ell\in[t',t'']$ for each $\ell\geq \ell_0$, $\ell\in\NN$.
Since $t',t''\in\Theta$ were arbitrary with $t'<\vartheta_{k_1+\dots+k_n}(\bk\otimes\bx)<t''$, this implies that $t_\ell$ converges to $\vartheta_{k_1+\dots+k_n}(\bk\otimes\bx)$ as $\ell\to\infty$.
\end{proof}

Next, we formulate a corollary of Theorem \ref{Thm_psi_est_eq_3_T}, which sheds light
 on the following feature of generalized $\psi$-estimators.
Given $n,m,\ell\in\NN$, suppose that one observes $\bz\in X^m$ and then $\by\in X^n$ repeatedly $\ell$-times.
In this case, roughly speaking, the generalized $\psi$-estimator does not essentially depend on $\bz\in X^m$ in the sense that its effect disappears as $\ell\to\infty$,
 for a detailed statement, see Corollary \ref{Cor_psi_est_eq_3_T} below.
One could say that, in this scenario, the observation $\bz$ can be considered asymptotically as an outlier.

\begin{Cor}\label{Cor_psi_est_eq_3_T}
Let $\psi\in\Psi[T](X,\Theta)$ and $n,m\in\NN$.
Then, for all $\by\in X^n$ and $\bz\in X^m$, we have
 \[
   \lim_{\ell\to\infty}\vartheta_{\ell n+m}(\ell\odot\by,\bz) = \vartheta_{n}(\by).
 \]
\end{Cor}

\begin{proof}
Let us apply Theorem \ref{Thm_psi_est_eq_3_T} with $\bx:=(\by,\bz)\in X^n\times X^m$,
 \[
  \bk:=\big( (1,\ldots,1),(0,\ldots,0) \big)\in\ZZ_+^n\times\ZZ_+^m
   \quad \text{and} \quad \bm:=\big( (0,\ldots,0),(1,\ldots,1)\big)\in\ZZ_+^n\times\ZZ_+^m.
 \]
Then, for each $\ell\in\NN$, we have
 \begin{align*}
      &\vartheta_{\ell(k_1+\dots+k_{n+m})+m_1+\dots+m_{n+m}}((\ell\bk+\bm)\otimes\bx)
        =  \vartheta_{\ell n+m}(  \ell\odot x_1,\ldots,\ell\odot x_n, x_{n+1},\ldots,x_{n+m} ) \\
      &\qquad = \vartheta_{\ell n+m}(  \ell\odot y_1,\ldots,\ell\odot y_n, z_1,\ldots,z_m  )
                    = \vartheta_{\ell n+m}(\ell\odot\by,\bz),
 \end{align*}
where the last equality follows from symmetry (see the paragraph after Definition \ref{Def_Tn}).
Hence Theorem \ref{Thm_psi_est_eq_3_T} yields that
  \[
    \lim_{\ell\to\infty}\vartheta_{\ell n+m}(\ell\odot\by,\bz)
      = \vartheta_{k_1+\dots+k_{n+m}}(\bk\otimes\bx)
      = \vartheta_{n}(x_1,\ldots,x_n)
      = \vartheta_{n}(y_1,\ldots,y_n)
      = \vartheta_{n}(\by).
  \]
\end{proof}

Note that, given $n,m\in\NN$, Corollary \ref{Cor_psi_est_eq_3_T} remains valid if, instead of $\psi\in\Psi[T](X,\Theta)$,
 we only suppose that $\psi$ has the properties $[T_n]$ and $[T_{\ell n+m}]$ for all $\ell\in\NN$.

\begin{Rem}\label{Rem1}
Corollary \ref{Cor_psi_est_eq_3_T} also implies Theorem \ref{Thm_psi_est_eq_3_T}.
Let $\psi\in\Psi[T](X,\Theta)$, $n\in\NN$, $\bk:=(k_1,\dots,k_n),\bm:=(m_1,\dots,m_n)\in\ZZ_+^n$ with $k_1+\dots+k_n>0$.
Applying Corollary \ref{Cor_psi_est_eq_3_T} with $\by:=\bk\otimes\bx\in X^{k_1+\cdots + k_n}$
 and $\bz:=\bm\otimes\bx\in X^{m_1+\cdots+m_n}$, we have
 \[
   \lim_{\ell\to\infty}\vartheta_{\ell(k_1+\dots+k_n)+m_1+\dots+m_n}( \ell\odot\by,\bz)
        =\vartheta_{k_1+\dots+k_n}(\by).
 \]
This yields Theorem \ref{Thm_psi_est_eq_3_T}, since, by symmetry (see the paragraph after Definition \ref{Def_Tn}),
 we have
 \begin{align*}
    \vartheta_{\ell(k_1+\dots+k_n)+m_1+\dots+m_n}( \ell\odot\by,\bz)
    & = \vartheta_{\ell(k_1+\dots+k_n)+m_1+\dots+m_n}( \ell\odot(\bk\otimes\bx),\bm\otimes\bx) \\
    & = \vartheta_{\ell(k_1+\dots+k_n)+m_1+\dots+m_n}( (\ell\bk +\bm) \otimes \bx).
 \end{align*}
\proofend
\end{Rem}

In the next result we establish the continuous dependence of the generalized $\psi$-estimators with respect to weights.

\begin{Thm}\label{Thm_psi_est_eq_3}
Let $n\in\NN$ and $\psi\in\Psi[W_n](X,\Theta)$.
Then, for all $\bx\in X^n$, the function
 \[
   \Lambda_n\ni\blambda\mapsto\vartheta_n^{\blambda}(\bx)\in\Theta
 \]
 is continuous.
\end{Thm}

\begin{proof}
Let $n\in\NN$ and $\bx=(x_1,\ldots,x_n)\in X^n$ be fixed arbitrarily,
 and define the function $f:\Lambda_n\to\Theta$ by $f(\blambda):=\vartheta_n^{\blambda}(\bx)$, $\blambda\in\Lambda_n$.
Then, we have
\begin{align}\label{help_equality_1}
   F(\blambda,t):=\sum_{i=1}^n \lambda_i\psi(x_i,t)
     \begin{cases}
        >0  & \text{if $t<f(\blambda)$, $t\in\Theta$,}\\
        <0  & \text{if $t>f(\blambda)$, $t\in\Theta$}
     \end{cases}
 \end{align}
 for all $\blambda:=(\lambda_1,\ldots,\lambda_n)\in\Lambda_n$.
Of course, we have $F(\blambda,t)=\psi_{\pmb{x},\pmb{\lambda}}(t)$, $t\in\Theta$, $\blambda\in\Lambda_n$, where $\psi_{\pmb{x},\pmb{\lambda}}$ is introduced in \eqref{help_mean_prop_csop}.
Observe that, for all $t\in\Theta$, the mapping $\Lambda_n\ni\blambda\mapsto F(\blambda,t)$ is continuous,
 i.e., $F$ is continuous in its first variable.

To prove the continuity of $f$ at a point $\blambda_0\in\Lambda_n$, it suffices to show that,
 for all sequences $(\blambda_k)_{k\in\NN}$ in $\Lambda_n$ converging to $\blambda_0$, we have that $f(\blambda_k)\to f(\blambda_0)$ as $k\to\infty$.
Let $\varepsilon>0$ be arbitrary and $(\blambda_k)_{k\in\NN}$ be an arbitrary sequence in $\Lambda_n$ converging to $\blambda_0$.
Choose $t',t''\in \Theta$ such that $f(\blambda_0)-\varepsilon<t'<f(\blambda_0) <t''<f(\blambda_0)+\varepsilon$.
Such choices of $t'$ and $t''$ are possible, since $f(\blambda_0)\in\Theta$ and $\Theta$ is an open set.
By \eqref{help_equality_1}, we have that $F(\blambda_0,t')>0>F(\blambda_0,t'')$.
By the continuity of $F$ in its first variable, we can find a $k_0\in\NN$ such that,
 for all $k\geq k_0$, the inequalities $F(\blambda_k,t')>0>F(\blambda_k,t'')$ are valid.
Using the definition of $f(\blambda_k)$, these inequalities yield that
$f(\blambda_0)-\varepsilon <t'\leq f(\blambda_k) \leq t''<f(\blambda_0)+\varepsilon$,
 i.e., $|f(\blambda_k)-f(\blambda_0)|<\varepsilon$ for $k\geq k_0$.
\end{proof}

\begin{Rem}
In case of $\psi\in\Psi[W](X,\Theta)$, Theorem \ref{Thm_psi_est_eq_3} implies Corollary \ref{Cor_psi_est_eq_3_T}
 (and hence, in view of Remark \ref{Rem1}, Theorem \ref{Thm_psi_est_eq_3_T} as well).
Let $n,m\in\NN$, $\by\in X^n$ and $\bz\in X^m$.
Since
 \[
    \sum_{i=1}^n \ell\psi(y_i,\vartheta) + \sum_{j=1}^n \psi(z_j,\vartheta)
      = \ell \left(\sum_{i=1}^n \psi(y_i,\vartheta) + \frac{1}{\ell}\sum_{j=1}^n \psi(z_j,\vartheta)\right),
      \qquad \vartheta\in\Theta,\;\; \ell\in\NN,
 \]
 we have
 \[
   \vartheta_{\ell n+m}(\ell\odot\by,\bz) = \vartheta_{n+m}^{\blambda_\ell}(\by,\bz), \qquad \ell\in\NN,
 \]
 where
 \[
   \blambda_\ell:= \Big(\underbrace{1,\ldots,1}_{n}, \underbrace{\frac{1}{\ell},\ldots,\frac{1}{\ell}}_{m}\Big)
                \to \big(\underbrace{1,\ldots,1}_{n}, \underbrace{0,\ldots,0}_{m}\big) \in\ZZ_+^{n+m} \qquad \text{as $\ell\to\infty$.}
 \]
Consequently, Theorem \ref{Thm_psi_est_eq_3} implies  Corollary \ref{Cor_psi_est_eq_3_T}.
\proofend
\end{Rem}

Next, we introduce the notion of sensitivity for estimators. We note that, in the field of means, some researchers call this property infinitesimality
 (for more details, see P\'ales \cite[page 67]{Pal1984}). Since the concept of infinitesimality has a standard meaning in probability theory,
 we decided to change the terminology that comes from the theory of means.

\begin{Def}\label{Def_sensitivity}
Let $X$ be a nonempty set, $\Theta$ be a nondegenerate open interval of $\RR$, and $M:\bigcup_{n=1}^\infty X^n\to\Theta$ be a function.
 We say that $M$ possesses the property of sensitivity (briefly, $M$ is sensitive) if, for all $x,y\in X$, $u,v\in\Theta$
 with $M_1(x)<u<v<M_1(y)$, there exist $k,m\in\NN$ such that
 \[
  u < M_{k+m}(k\odot x, m\odot y) < v,
 \]
 where, for each $n\in\NN$, $M_n$ denotes the restriction of $M$ to $X^n$.
\end{Def}

Next, we prove that the property sensitivity holds for generalized $\psi$-estimators.
In fact, we prove a more general result by generalizing Theorem 3 in P\'ales \cite{Pal1984}, which is about quasi-deviation means.

\begin{Thm}\label{Thm_psi_est_infinit}
Let $N\in\NN$, let $\Theta_0,\Theta_1,\dots,\Theta_N$ be nondegenerate open intervals of $\RR$. For $i\in\{1,\dots,N\}$, let $\psi_i\in\Psi[T,W_2](X,\Theta_i)$ and let $g:\Theta_1\times\dots\times\Theta_N\to\Theta_0$ be a continuous function. Then the function $M:\bigcup_{n=1}^\infty X^n\to\Theta_0$ defined by
 \[
    M(x_1,\ldots,x_n) := g\big( \vartheta_{n,\psi_1}(x_1,\ldots,x_n), \ldots, \vartheta_{n,\psi_N}(x_1,\ldots,x_n)\big),
     \qquad n\in\NN,\quad x_1,\ldots,x_n\in X,
 \]
possesses the property of sensitivity.
\end{Thm}

\begin{proof}
First, we show that if $\psi\in\Psi[T,W_2](X,\Theta)$ (where $\Theta$ is a nondegenerate open interval of $\RR$),
 then, for all $x,y\in X$, define $e:[0,1]\to\Theta$ by
\[
  e(\lambda):=\vartheta_{2,\psi}^{(1-\lambda,\lambda)}(x,y), \qquad \lambda\in[0,1].
\]
By applying Theorem \ref{Thm_psi_est_eq_3} for $n=2$,
 we have that the function $e$ is continuous.
On the other hand, using the null-homogeneity of the weighted generalized $\psi$-estimators with respect to the weights (see the discussion after Definition \ref{Def_Tn}), for all $k,m\in\ZZ_+$ with $k+m>0$, we can obtain
 \[
  \vartheta_{k+m,\psi}(k\odot x,m\odot y)
  =\vartheta_{2,\psi}^{(k,m)}(x,y)
  =\vartheta_{2,\psi}^{\left(\frac{k}{k+m},\frac{m}{k+m}\right)}(x,y)
  =e\left(\frac{m}{k+m}\right).
 \]

In what follows, let $x,y\in X$, $u,v\in\Theta_0$ with $M_1(x)<u<v<M_1(y)$ be fixed arbitrarily.
Then, as we have just proved, for each $i\in\{1,\ldots,N\}$, there exists a continuous function $e_i:[0,1]\to\Theta_i$ such that
 \[
  \vartheta_{k+m,\psi_i}(k\odot x,m\odot y) = e_i\left(\frac{m}{k+m}\right)
    \qquad \text{for all \ $k,m\in\ZZ_+$ \ with $k+m>0$,}
 \]
 and $e_i(\lambda)=\vartheta_{2,\psi_i}^{(1-\lambda,\lambda)}(x,y)$, $\lambda\in[0,1]$.
Define $f:[0,1]\to\Theta_0$ by
 \[
 f(\lambda):=g(e_1(\lambda),\ldots,e_N(\lambda)),\qquad \lambda\in[0,1].
 \]
Then $f$ is continuous (being a composition of continuous functions), and
 \begin{align*}
 &f(0):= g(e_1(0),\ldots,e_N(0))
      = g\big(\vartheta_{2,\psi_1}^{(1,0)}(x,y), \ldots, \vartheta_{2,\psi_N}^{(1,0)}(x,y) \big)\\
 &\phantom{f(0)\;}
      = g\big(\vartheta_{1,\psi_1}(x), \ldots, \vartheta_{1,\psi_N}(x) \big)
      = M_1(x),\\
 &f(1):= g(e_1(1),\ldots,e_N(1))
      = g\big(\vartheta_{2,\psi_1}^{(0,1)}(x,y), \ldots, \vartheta_{2,\psi_N}^{(0,1)}(x,y) \big)\\
 &\phantom{f(1)\;}
      = g\big(\vartheta_{1,\psi_1}(y), \ldots, \vartheta_{1,\psi_N}(y) \big)
      = M_1(y).
 \end{align*}
Due to the continuity of $f$, we have that $f([0,1])\supseteq [M_1(x),M_1(y)]\supseteq (u,v)$. Thus, $f^{-1}((u,v))$ is a nonempty and open subset of $(0,1)$. Hence, there exists a rational number $t\in(0,1)$ such that $f(t)\in(u,v)$.
By representing $t$ in the form $t=\frac{m}{k+m}$ with some $k,m\in\ZZ_+$ satisfying $k+m>0$, we have
 \begin{align*}
  M_{k+m}((k\odot x, m\odot y))
    & = g\big(\vartheta_{k+m,\psi_1}(k\odot x,m\odot y), \ldots, \vartheta_{k+m,\psi_N}(k\odot x,m\odot y) \big) \\
    & = g\left(  e_1\left(\frac{m}{k+m}\right),\ldots,  e_N\left(\frac{m}{k+m}\right)  \right)
      = f\left(\frac{m}{k+m}\right)
      = f(t)\in(u,v).
 \end{align*}
This implies that $M$ is sensitive, as desired.
\end{proof}

Theorem \ref{Thm_psi_est_infinit} readily yields the following corollary by choosing $N=1$ and $g(t)=t$, $t\in\Theta$.

\begin{Cor}\label{Cor_infinit}
Let $\psi\in\Psi[T,W_2](X,\Theta)$. Then the generalized $\psi$-estimator
 $\vartheta_\psi:\bigcup_{n=1}^\infty X^n\to\Theta$ defined by
 \[
    \vartheta_\psi(x_1,\ldots,x_n):= \vartheta_{n,\psi}(x_1,\ldots,x_n),
     \qquad n\in\NN,\quad x_1,\ldots,x_n\in X,
 \]
 possesses the property of sensitivity.
\end{Cor}

\section{Examples for Sections \ref{section_ine_mon_prop} and \ref{section_asy_con_prop}}
\label{section_examples}

In the first two examples, we demonstrate some possible applications of Theorems \ref{Thm_psi_becsles_mean_prop} and \ref{Thm_bisymmetry2}.

\begin{Ex}\label{Ex2}
Theorem \ref{Thm_psi_becsles_mean_prop} can efficiently be used to show that a certain estimator cannot be a generalized $\psi$-estimator.
To demonstrate this, we apply the example presented on page 66 of P\'ales \cite{Pal1984}.
Let $X:=\Theta:=\RR_{++}$ and define the estimator $\kappa:\bigcup_{n=1}^\infty\RR_{++}^n\to\RR_{++}$ by
\[
  \kappa(x_1,\dots,x_n):=\frac1{2n}\Big(x_1+\dots+x_n+n\sqrt[n]{x_1\cdots x_n}\Big),
  \qquad n\in\NN,\,\,x_1,\dots,x_n\in\RR_{++}.
\]
Then we claim that there does not exist $\psi\in\Psi[T](\RR_{++},\RR_{++})$ such that
\Eq{help_ex1}{
  \kappa(x_1,\dots,x_n)=\vartheta_{n,\psi}(x_1,\dots,x_n),\qquad n\in\NN,\,\,x_1,\dots,x_n\in\RR_{++}.
}
To the contrary, assume that \eqref{help_ex1} holds for some $\psi\in\Psi[T](\RR_{++},\RR_{++})$.
Then applying Theorem \ref{Thm_psi_becsles_mean_prop} to the $4$-tuple $(x_1,x_2,x_3,x_4):=(1,81,25,25)$, we would obtain that
\[
 \min(\vartheta_{2,\psi}(1,81),\vartheta_{2,\psi}(25,25))\leq \vartheta_{4,\psi}(1,81,25,25)\leq
 \max(\vartheta_{2,\psi}(1,81),\vartheta_{2,\psi}(25,25)).
\]
Therefore
\Eq{help_ex}{
 \min(\kappa(1,81),\kappa(25,25))\leq \kappa(1,81,25,25)\leq \max(\kappa(1,81),\kappa(25,25)).
}
On the other hand, one can easily compute that
\[
 \kappa(1,81)=\kappa(25,25)=25 \qquad\mbox{and}\qquad
 \kappa(1,81,25,25)= 24,
\]
and hence the left hand side inequality in \eqref{help_ex} is violated, which proves that \eqref{help_ex1} cannot be valid even for $n=4$.
\proofend
\end{Ex}

\begin{Ex}\label{Ex4}
Theorem \ref{Thm_bisymmetry2} can also be used to show that a given estimator cannot be a generalized $\psi$-estimator.
To see this, we consider the estimator $\kappa$ introduced in Example \ref{Ex2}.
Then we claim that there does not exist $\psi\in\Psi[T_4](\RR_{++},\RR_{++})$ such that
\Eq{*}{
  \kappa(x_1,\dots,x_n)=\vartheta_{n,\psi}(x_1,\dots,x_n),\qquad n\in\NN,\,\,x_1,\dots,x_n \in\RR_{++}.
}
To the contrary, assume that this equality holds for some $\psi\in\Psi[T_4](\RR_{++},\RR_{++})$.
Then applying Theorem \ref{Thm_bisymmetry2} with the choices  $X:=\Theta:=\RR_{++}$, $n:=4$, $m:=2$ and
 \[
     \begin{bmatrix}
       x_{1,1} & x_{1,2} \\
       x_{2,1} & x_{2,2} \\
       x_{3,1} & x_{3,2} \\
       x_{4,1} & x_{4,2} \\
     \end{bmatrix}
     := \begin{bmatrix}
       1 & 81 \\
       81 & 1 \\
       25 & 25 \\
       25 & 25 \\
     \end{bmatrix} \in\RR_{++}^{4\times 2},
 \]
 we would have that
\begin{align*}
 &\min\big(\vartheta_{2,\psi}(1,81),\vartheta_{2,\psi}(81,1),\vartheta_{2,\psi}(25,25),\vartheta_{2,\psi}(25,25)\big) \\
 &\qquad \leq \max\big(\vartheta_{4,\psi}(1,81,25,25), \vartheta_{4,\psi}(81,1,25,25)\big).
\end{align*}
Therefore
 \[
  \min\big(\kappa(1,81),\kappa(81,1),\kappa(25,25),\kappa(25,25)\big)
    \leq
  \max\big(\kappa(1,81,25,25), \kappa(81,1,25,25)\big).
 \]
This leads us to a contradiction, since
 \[
   \kappa(1,81) = \kappa(81,1)=\kappa(25,25) = 25,
   \qquad
  \kappa(1,81,25,25) = \kappa(81,1,25,25) = 24.
 \]

We note that $\kappa$ satisfies the bisymmetry-type inequality in Theorem \ref{Thm_bisymmetry2} for $n=m=2$.
To see this, observe that
\[
 \kappa(x,y)=\frac14(x+y+2\sqrt{xy})
 =\big(\tfrac12(\sqrt{x}+\sqrt{y})\big)^2, \qquad x,y\in\RR_{++}.
\]
On the other hand, let $\psi:\RR_{++}\times\RR_{++}\to\RR$ be given by $\psi(x,t):=\sqrt{x}-\sqrt{t}$, $x,t\in\RR_{++}$.
Then $\psi\in\Psi[T](\RR_{++},\RR_{++})$, and one can easily see that
 \[
    \vartheta_{n,\psi}(x_1,\ldots,x_n) = \left(\frac{1}{n}\sum_{i=1}^n \sqrt{x_i}\right)^2,
     \qquad n\in\NN, \;\; x_1,\ldots,x_n\in\RR_{++}.
 \]
In particular, we have $\vartheta_{2,\psi}(x,y)=\kappa(x,y)$, $x,y\in\RR_{++}$.
However, it does not hold that $\vartheta_{3,\psi}(x,y,z)=\kappa(x,y,z)$, $x,y,z\in\RR_{++}$,
 since, for example, $\vartheta_{3,\psi}(1,1,64)=((1+1+8)/3)^2 = 100/9$, but
 $\kappa(1,1,64)=(1+1+64+3\cdot 4)/6 = 78/6=13$.
In view of Theorem \ref{Thm_bisymmetry2}, $\psi$ satisfies the bisymmetry-type inequality, and hence we have that
\[
 \min(\vartheta_{2,\psi}(x,y),\vartheta_{2,\psi}(u,v))\leq \max(\vartheta_{2,\psi}(x,u),\vartheta_{2,\psi}(y,v)), \qquad x,y,u,v\in\RR_{++}.
\]
Using the identity $\vartheta_{2,\psi}(x,y)=\kappa(x,y)$, $x,y\in\RR_{++}$, this relation shows that $\kappa$ satisfies the 2-variable bisymmetry-type inequality.
\proofend
\end{Ex}

In the next example, we shed some light on the conditions of Theorem \ref{Thm_theta_monotone}.

\begin{Ex}
Let $X:=\RR$, $\Theta:=\RR$ and  $\psi: X\times \Theta\to\RR$ be given by
 \[
  \psi(x,t):= \sign(x-t), \qquad x,t\in\RR.
 \]
Then $\psi\in\Psi[T_1](X,\Theta)$, and $\vartheta_1(x)=x$, $x\in X$.
We check that $\psi$ is a $T_2^{(\lambda,1-\lambda)}$-function for any $\lambda\in[0,1]\setminus \{\frac{1}{2}\}$, but it is not a $T_2^{(\frac{1}{2},\frac{1}{2})}$-function.

If $\lambda\in[0,1]\setminus \{\frac{1}{2}\}$, then $\psi$ is a $T_2^{(\lambda,1-\lambda)}$-function with
\begin{align}\label{eq:help_0}
   \vartheta_2^{(\lambda,1-\lambda)}(x,y)
     := \begin{cases}
           y & \text{if $\lambda\in[0,\frac{1}{2})$,}\\
           x & \text{if $\lambda\in(\frac{1}{2},1]$,}
        \end{cases} \qquad x,y\in\RR.
\end{align}
Indeed, if $x=y$, then $\lambda\psi(x,t) + (1-\lambda)\psi(y,t) = \psi(x,t)$, $t\in\RR$.
If $x<y$, then
 \begin{align*}
  \lambda\psi(x,t) + (1-\lambda)\psi(y,t)
      = \begin{cases}
          1 & \text{if $t<x$,}\\
          1-\lambda & \text{if $t=x$,}\\
          1-2\lambda  & \text{if $x<t<y$,}\\
          -\lambda  & \text{if $t=y$,}\\
          -1 & \text{if $t>y$,}
        \end{cases}
 \end{align*}
while if $x>y$, then
 \begin{align*}
  \lambda\psi(x,t) + (1-\lambda)\psi(y,t)
      = \begin{cases}
          1 & \text{if $t<y$,}\\
          \lambda & \text{if $t=y$,}\\
          2\lambda-1 & \text{if $y<t<x$,}\\
          \lambda-1  & \text{if $t=x$,}\\
          -1 & \text{if $t>x$.}
        \end{cases}
 \end{align*}
These equalities show that \eqref{eq:help_0} is valid if
$\lambda\in[0,1]\setminus \{\frac{1}{2}\}$.

On the other hand, if $x,y\in\RR$ with $x<y$, then the function $\RR\ni t \mapsto \frac{1}{2}\psi(x,t) + \frac{1}{2}\psi(y,t)$ takes value 0 on the non-degenerate open interval $(x,y)$, and hence $\psi$ is not a $T_2^{(\frac{1}{2},\frac{1}{2})}$-function.

One can also see that the function $[0,1]\setminus \{\frac{1}{2}\} \ni \lambda \mapsto \vartheta_2^{(\lambda,1-\lambda)}(x,y)$
 is decreasing in case of $x\leq y$, and is increasing in case of $x\geq y$, but not strictly decreasing/increasing in the corresponding cases.
This example might suggest that, in general, one cannot expect that the function $\cJ_{\ba,\bb}\ni s\mapsto \vartheta_n^{s\ba+\bb}(\bx)\in\Theta$ in
Theorem \ref{Thm_theta_monotone} is {\sl strictly} monotone.
\proofend
\end{Ex}

In what follows, all the random variables are defined on an appropriate probability space $(\Omega,\cA,\PP)$,
 and, given a random variable $\xi$, by a sample of size $n$ for $\xi$ (where $n\in\NN$), we mean i.i.d.\ random variables
 \ $\xi_1,\ldots,\xi_n$ \ with the same distribution as that of $\xi$.
In the forthcoming three examples, the following two settings will repeatedly come into play.

 \begin{enumerate}[{\sc Setting} I:]
    \item
     Let $\xi$ be a normally distributed random variable with mean $m\in\RR$ and with variance $\sigma^2$, where $\sigma>0$.
     We suppose that $\sigma$ is known.
     Let $n\in\NN$ and $x_1,\ldots,x_n\in\RR$ be a realization of a sample of size $n$ for $\xi$.
     It is known that there exists a unique MLE of $m$ based on $x_1,\ldots,x_n\in\RR$, and it takes the form $\widehat m_n:=\frac{x_1+\cdots+x_n}{n}$.
     By Barczy and P\'ales \cite[Example 4.5]{BarPal2}, the corresponding function $\psi$ takes the form $\psi:\RR\times\RR\to\RR$,
     \[
       \psi(x,m) = \frac{1}{\sigma^2}(x-m),\qquad x,m\in\RR.
    \]
   \item
    Let $\alpha>0$ and let $\xi$ be an absolutely continuous random variable with a density function
       \begin{align*}
        f_\xi(x):=\begin{cases}
                2\alpha x (1-x^2)^{\alpha-1}  & \text{if $x\in(0,1)$,}\\
                0 & \text{otherwise.}
       \end{cases}
  \end{align*}
    Then one can check that given $n\in\NN$ and a realization $x_1,\ldots,x_n\in(0,1)$
    of a sample of size $n$ for $\xi$, there exists a unique MLE of $\alpha$ and it takes the form
    \begin{align*}
      \widehat\alpha_n:=-\frac{n}{\sum_{i=1}^n \ln(1-x_i^2)}.
    \end{align*}
    By Barczy and P\'ales \cite[Example 4.6]{BarPal2},  the corresponding function $\psi$ takes the form $\psi:(0,1)\times(0,\infty)\to\RR$,
   \[
     \psi(x,\alpha) = \frac{1}{\alpha} + \ln(1-x^2), \qquad x\in(0,1),\quad \alpha>0.
   \]
  \end{enumerate}

Next, we present some examples for sensitive and non-sensitive functions $M$, which have important role in statistics.

\begin{Ex}\label{Ex5}
(i):\ Assume Setting I.
By choosing $X:=\Theta:=\RR$, the function $M:\bigcup_{n=1}^\infty X^n\to\Theta$, defined by
 $M(x_1,\ldots,x_n):=\frac{1}{n}(x_1+\cdots+x_n)$, $n\in\NN$, $(x_1,\ldots,x_n)\in X^n$, possesses the property of sensitivity.
Indeed, let $x,y,u,v\in\RR$  be given such that $M_1(x)=x<u<v<y=M_1(y)$.
Using that the function $e:[0,1]\to\RR$, $e(\lambda):=\lambda x + (1-\lambda)y$, $\lambda\in[0,1]$, is continuous,
 $e(0)=y$ and $e(1)=x$, by Bolzano's theorem, there exists $\lambda_0\in(0,1)$ such that $e(\lambda_0)\in(u,v)$.
Using again the continuity of $e$, there exist $p,q\in\NN$ such that $p<q$ and $e\big(\frac{p}{q}\big)\in(u,v)$.
By choosing $k:=p$ and $m:=q-p$, we have
 \[
   e\left(\frac{p}{q}\right) = e\left(\frac{k}{k+m}\right) =\frac{k}{k+m}x + \frac{m}{k+m}y = M_{k+m}(k\odot x, m\odot y),
 \]
 yielding that $M_{k+m}(k\odot x, m\odot y)\in(u,v)$.

(ii): Assume Setting II.
By choosing $X:=(0,1)$ and $\Theta:=(0,\infty)$, the function $M:\bigcup_{n=1}^\infty X^n\to\Theta$, defined by
 $M(x_1,\ldots,x_n):= - n/\sum_{i=1}^n \ln(1-x_i^2)$, $n\in\NN$, $(x_1,\ldots,x_n)\in X^n$, possesses the property of sensitivity.
Indeed, let $x,y\in(0,1)$ and $u,v\in(0,\infty)$  be given such that
 \[
   M_1(x)=-\frac{1}{\ln(1-x^2)}<u<v<-\frac{1}{\ln(1-y^2)}=M_1(y).
 \]
Define the function $e:[0,1]\to\RR$ by
\[
  e(\lambda):=-\frac{1}{\lambda \ln(1-x^2) + (1-\lambda) \ln(1-y^2)},
  \qquad \lambda\in[0,1].
\]
Then $e$ is continuous, $e(0) = M_1(y)$ and $e(1) = M_1(x)$.
By Bolzano's theorem, there exists $\lambda_0\in(0,1)$ such that $e(\lambda_0)\in(u,v)$. Using again the continuity of $e$, there exist $p,q\in\NN$ such that $p<q$ and $e\big(\frac{p}{q}\big)\in (u,v)$. By choosing $k:=p$ and $m:=q-p$, we have
 \[
   e\left(\frac{p}{q}\right) =
   -\frac{k+m}{k\ln(1-x^2) + m\ln(1-y^2)}
   = M_{k+m}(k\odot x, m\odot y).
 \]
Consequently, $M_{k+m}(k\odot x, m\odot y)\in(u,v)$, as desired.

(iii):
Let $X:=\Theta:=\RR$, and define $M:\bigcup_{n=1}^\infty X^n\to\Theta$ by
 $M(x_1,\ldots,x_n):=\max(x_1,\ldots,x_n)$, $n\in\NN$, $(x_1,\ldots,x_n)\in X^n$.
We show that $M$ does not possess the property of sensitivity.
Let $x,y,u,v\in\RR$ be given such that $M_1(x)=x<u<v<y=M_1(y)$.
Then for each $k,m\in\NN$, we have $M_{k+m}(k\odot x, m\odot y) = y$, and hence $M_{k+m}(k\odot x, m\odot y)<v$ cannot hold,
 yielding that $M$ is not sensitive.
According to Corollary \ref{Cor_infinit}, the fact that $M$ does not possess the property of sensitivity also implies that $M$ cannot be derived as a generalized $\psi$-estimator with some $\psi\in\Psi[T,W_2](\RR,\RR)$.

Similarly, we can check that the function $\widetilde M:\bigcup_{n=1}^\infty X^n\to\Theta$ defined by
 \[
  \widetilde M(x_1,\ldots,x_n):=\frac12(\min(x_1,\ldots,x_n)+\max(x_1,\ldots,x_n)),\qquad  n\in\NN, \quad (x_1,\ldots,x_n)\in X^n,
 \]
does not possess the property of sensitivity (hence, it is also not generated by any $\psi\in\Psi[T,W_2](\RR,\RR)$ as a generalized $\psi$-estimator).
Let $x,y,u,v\in\RR$ be given such that $\widetilde M_1(x)=x<u<v<y=\widetilde M_1(y)$.
Then for each $k,m\in\NN$, we have $\widetilde M_{k+m}(k\odot x, m\odot y) = \frac{x+y}{2}$.
Consequently, in the case of $x<\frac{x+y}{2}<u<v<y$, we have that $u<\widetilde M_{k+m}(k\odot x, m\odot y)$ cannot be valid,
 yielding that $\widetilde M$ is not sensitive.

The estimators $M$ and $\widetilde M$ have statistical relevance.
Recall that if $\xi$ is a random variable uniformly distributed on the interval $[a,b]$, where $a<b$, and $x_1,\ldots,x_n\in[a,b]$ is a realization of a sample of size $n$ for $\xi$ (where $n\in\NN$),
then it is known that there exists a unique MLE of the parameters $a$ and $b$ based on $x_1,\ldots,x_n$,
and it takes the form $\min(x_1,\ldots,x_n)$ and $\max(x_1,\ldots,x_n)$, respectively. Further, $\frac12(\min(x_1,\ldots,x_n)+\max(x_1,\ldots,x_n))$ is the MLE of mean $\frac{a+b}{2}$ of $\xi$, which is called the sample mid range.
\proofend
\end{Ex}

In the next example we calculate the weighted generalized $\psi$-estimators for two particular functions $\psi$,
 and then we formulate the 2-variable bisymmetry-type inequality in these special cases.

\begin{Ex}\label{Ex1}
(i): Assume Setting I.
Then for each $n\in\NN$, $x_1,\ldots,x_n\in\RR$ and $(\lambda_1,\ldots,\lambda_n)\in\Lambda_n$,
 the equation $\sum_{i=1}^n \lambda_i \psi(x_i,m)=0$ has a unique solution $m\in\RR$, which is given by
 \begin{align}\label{help_ex_1}
   \vartheta_{n}^{(\lambda_1,\ldots,\lambda_n)}(x_1,\ldots,x_n)
     = \frac{\sum_{i=1}^n \lambda_i x_i}{\sum_{j=1}^n \lambda_j},
 \end{align}
 and hence $\psi$ is a $W$-function.
Note that $\vartheta_{n}^{(\lambda_1,\ldots,\lambda_n)}(x_1,\ldots,x_n)$ is nothing else but
 the weighted arithmetic mean of $x_1,\ldots,x_n$ with weights $\lambda_i$, $i=1,\ldots,n$.
In this case, the 2-variable bisymmetry-type inequality \eqref{ineq_22_bisymmeytry} takes the form
 \begin{align}\label{help_16}
  \begin{split}
  &\min\left(\frac{\alpha}{\alpha+\beta}x+\frac{\beta}{\alpha+\beta}y,
              \frac{\gamma}{\gamma+\delta}u+\frac{\delta}{\gamma+\delta}v \right) \\
  &\qquad \leq  \max\left(\frac{\alpha}{\alpha+\gamma}x+\frac{\gamma}{\alpha+\gamma}u,
              \frac{\beta}{\beta+\delta}y+\frac{\delta}{\beta+\delta}v \right)
  \end{split}
 \end{align}
 for all $x,y,u,v\in \RR$ and $\alpha,\beta,\gamma,\delta\in\RR_+$ such that $\alpha+\beta>0$,
 $\gamma+\delta>0$, \ $\alpha+\gamma>0$ and $\beta+\delta>0$.

We remark that the inequality \eqref{help_16} can also be directly checked using only the definition of weighted arithmetic means.
Namely, the left hand side of \eqref{help_16} can be upper estimated as follows
 \begin{align*}
    &\min( \vartheta_2^{(\alpha,\beta)}(x,y), \vartheta_2^{(\gamma,\delta)}(u,v) )\\
    &\leq \vartheta_2^{(\alpha+\beta, \gamma+\delta)} \big(\vartheta_2^{(\alpha,\beta)}(x,y),  \vartheta_2^{(\gamma,\delta)}(u,v)\big)\\
     &= \frac{\alpha+\beta}{\alpha+\beta+\gamma+\delta} \left( \frac{\alpha}{\alpha+\beta}x + \frac{\beta}{\alpha+\beta}y \right)
        + \frac{\gamma+\delta}{\alpha+\beta+\gamma+\delta} \left( \frac{\gamma}{\gamma+\delta}u + \frac{\delta}{\gamma+\delta}v \right)\\
     &= \frac{1}{\alpha+\beta+\gamma+\delta} (\alpha x+ \beta y + \gamma u + \delta v)\\
     &= \frac{\alpha+\gamma}{\alpha+\gamma+\beta+\delta} \left( \frac{\alpha}{\alpha+\gamma}x + \frac{\gamma}{\alpha+\gamma}u \right)
        + \frac{\beta+\delta}{\alpha+\gamma+\beta+\delta} \left( \frac{\beta}{\gamma+\delta}y + \frac{\delta}{\beta+\delta}v \right)\\
     &= \vartheta_2^{(\alpha+\gamma, \beta+\delta)} \big(\vartheta_2^{(\alpha,\gamma)}(x,u),  \vartheta_2^{(\beta,\delta)}(y,v)\big)\\
     &\leq \max( \vartheta_2^{(\alpha,\gamma)}(x,u), \vartheta_2^{(\beta,\delta)}(y,v) ),
 \end{align*}
where the two inequalities follow by the definition of 2-variable means.

(ii): Assume Setting II.
Then for each $n\in\NN$, $x_1,\ldots,x_n\in(0,1)$ and $(\lambda_1,\ldots,\lambda_n)\in\Lambda_n$,
 the equation $\sum_{i=1}^n \lambda_i \psi(x_i,\alpha)=0$, $\alpha>0$, has the unique solution
 \begin{align}\label{help_ex_2}
  \begin{split}
     \vartheta_{n}^{(\lambda_1,\ldots,\lambda_n)}(x_1,\ldots,x_n)
   = -\frac{\sum_{i=1}^n \lambda_i}{\sum_{i=1}^n \lambda_i \ln(1-x_i^2)}
   = - \frac{1}{\ln\left( \prod_{i=1}^n (1-x_i^2)^{\frac{\lambda_i}{\sum_{j=1}^n \lambda_j}} \right)},
 \end{split}
 \end{align}
and hence $\psi$ is a $W$-function.
In this case the 2-variable bisymmetry-type inequality \eqref{ineq_22_bisymmeytry} takes the form
  \begin{align*}
  &\min\left(-\frac{a+b}{a \ln(1-x^2) + b \ln(1-y^2)},
                -\frac{c+d}{c \ln(1-u^2) + d \ln(1-v^2)}  \right) \\
  &\qquad \leq  \max\left( -\frac{a+c}{a \ln(1-x^2) + c \ln(1-u^2)},
                           -\frac{b+d}{b \ln(1-y^2) + d \ln(1-v^2)} \right)
 \end{align*}
 for all $x,y,u,v\in(0,1)$ and $a,b,c,d\in\RR_+$ such that $a+b>0$,
 $c+d>0$, \ $a+c>0$ and $b+d>0$.
\proofend
\end{Ex}

In the next example, we give applications of Corollary \ref{Cor_psi_est_eq_3_T}.

\begin{Ex}
(i): Assume Setting I.
By \eqref{help_ex_1}, for all $n,m\in\NN$, $\by=(y_1,\dots,y_n)\in \RR^n$ and $\bz=(z_1,\dots,z_m)\in \RR^m$, we have
 \[
   \vartheta_{\ell n+m}(\ell\odot\by,\bz)
      = \frac{\ell(y_1+\dots+y_n)+z_1+\dots+z_m}{\ell n+m},\qquad \ell\in\NN,
 \]
 and hence
 \[
   \vartheta_{\ell n+m}(\ell\odot\by,\bz) \to \frac{y_1+\dots+y_n}{n} = \vartheta_n(\by)
     \qquad \text{as $\ell\to\infty$.}
 \]
This is in accordance with Corollary \ref{Cor_psi_est_eq_3_T}.

(ii): Assume Setting II.
By \eqref{help_ex_2}, for all $\by=(y_1,\dots,y_n)\in (0,1)^n$ and $\bz=(z_1,\dots,z_m)\in(0,1)^m$, we have
 \[
   \vartheta_{\ell n+m}(\ell\odot\by,\bz)
     = - \frac{1}{\ln\left(\prod_{i=1}^n(1-y_i^2)^{\frac{\ell}{\ell n+m}}
     \prod_{j=1}^m(1-z_j^2)^{\frac{1}{\ell n+m}} \right)},
     \qquad \ell\in\NN,
 \]
 and hence
 \[
    \vartheta_{\ell n+m}(\ell\odot\by,\bz)
    \to -\frac{1}{\ln\left(\prod_{i=1}^n(1-y_i^2)^{\frac{1}{n}} \right)} = \vartheta_n(\by)
     \qquad \text{as \ $\ell\to\infty$.}
 \]
 This is in accordance with Corollary \ref{Cor_psi_est_eq_3_T}.
\proofend
\end{Ex}

\appendix

\section{Appendix: quasi-affine functions on real line}\label{App1}

In this part, we give a characterisation of quasi-affine functions defined on a subinterval of $\RR$.
 This can be considered as a folklore in the literature, but we could not address any reference for it,
  so we decided to provide a proof as well.
First, we present an auxiliary lemma on monotone functions, which is again a folklore in the literature.

\begin{Lem}\label{Lem_monotone}
Let $I\subseteq \RR$ be a nondegenerate interval and $f:I\to\RR$.
Then $f$ is monotone if and only if there do not exist $x,y,z\in I$ with $x<y<z$ such that $f(x)<f(y)$ and $f(y)>f(z)$,
 or $f(x)>f(y)$ and $f(y)<f(z)$ hold.
\end{Lem}

\begin{proof}
Assume that $f$ is monotone.
Then for all $x,y,z\in I$ with $x<y<z$, we have $f(x)\leq f(y)\leq f(z)$ or $f(x)\geq f(y)\geq f(z)$.
Hence there do not exist $x,y,z\in I$ with $x<y<z$ such that $f(x)<f(y)$ and $f(y)>f(z)$,
 or $f(x)>f(y)$ and $f(y)<f(z)$ hold.

Now assume that there do not exist $x,y,z\in I$ with $x<y<z$ such that $f(x)<f(y)$ and $f(y)>f(z)$,
 or $f(x)>f(y)$ and $f(y)<f(z)$ hold.
We will frequently use that this assumption implies that
 if $s,t\in I$ with $s<t$ such that $f(s)\leq f(t)$, then $f(r)\in[f(s),f(t)]$ for all $r\in[s,t]$.
If $f$ is monotone decreasing, then we are done.
If $f$ is not monotone decreasing, then there exist $a,b\in I$ with $a<b$ such that $f(a)<f(b)$. Then
we are going to show that $f$ is monotone increasing. To see this, let $x,y\in I$ be arbitrarily fixed with $x<y$.
We can distinguish six cases.

If $x<y\leq a$, then $f(x)\leq f(a)$ (otherwise $f(x)>f(a)$ and $f(a)<f(b)$ would lead us to a contradiction),
 and hence $f(y)\in[f(x),f(a)]$, yielding that $f(x)\leq f(y)$.

If $x\leq a < y \leq b$, then $f(x)\leq f(a)$ and $f(y)\in[f(a),f(b)]$, yielding that $f(x)\leq f(y)$.

If $x\leq a < b < y$, then $f(x)\leq f(a)$ and $f(b)\leq f(y)$ (otherwise $f(b)>f(y)$ and $f(a)<f(b)$ would lead us to a contradiction),
 yielding that $f(x)\leq f(y)$.

If $a<x<y\leq b$, then $f(x)\in[f(a),f(b)]$ and hence $f(y)\in[f(x),f(b)]$, yielding that $f(x)\leq f(y)$.

If $a<x\leq b<y$, then $f(x)\in [f(a),f(b)]$ and $f(b)\leq f(y)$ (otherwise $f(b)>f(y)$ and $f(b)>f(a)$ would lead us to a contradiction),
  yielding that $f(x)\leq f(y)$.

If $a<b<x<y$, then $f(x)\geq f(b)$  (otherwise $f(x)<f(b)$ and $f(b)>f(a)$ would lead us to a contradiction)
 and hence $f(y)\geq f(x)$ (otherwise $f(y)<f(x)$ and $f(x)\geq f(b)>f(a)$ would lead us to a contradiction), i.e.,  $f(x)\leq f(y)$.

All in all, $f(x)\leq f(y)$ for all $x,y\in I$ with $x<y$, i.e., $f$ is monotone increasing.
\end{proof}

\begin{Pro}\label{Pro_quasi}
Let $I\subseteq \RR$ be a nondegenerate interval and $f:I\to\RR$. Then $f$ is quasi-affine if and only if $f$ is monotone on $I$.
\end{Pro}

\begin{proof}
Assume that $f$ is quasi-affine, and, on the contrary, suppose that $f$ is not monotone.
Then, by Lemma \ref{Lem_monotone}, there exist $x,y,z\in I$ with $x<y<z$ such that $f(x)<f(y)$ and $f(y)>f(z)$,
 or $f(x)>f(y)$ and $f(y)<f(z)$ hold.
Since $x<y<z$, there exists $t\in(0,1)$ such that $y=tx+(1-t)z$.

\noindent In the case of $f(x)<f(y)$ and $f(y)>f(z)$, we have
 \[
   f(tx+(1-t)z) = f(y) > \max(f(x),f(z)),
 \]
 and hence $f$ cannot be quasi-convex, leading us to a contradiction.

\noindent In the case of $f(x)>f(y)$ and $f(y)<f(z)$, we have
 \[
   f(tx+(1-t)z) = f(y) < \min(f(x),f(z)),
 \]
 and hence $f$ cannot be quasi-concave, leading us to a contradiction as well.

Assume now that $f$ is monotone.
Then
 \[
  f(tx+(1-t)y) \in\big[ \min(f(x),f(y)), \max(f(x),f(y))\big], \qquad x,y\in I, \;\; t\in(0,1),
 \]
 yielding that $f$ is quasi-convex and quasi-concave, i.e., quasi-affine.
\end{proof}

\small


\begin{thebibliography}{10}

\bibitem{Baj58}
M.~Bajraktarevi{\'c}.
\newblock Sur une {\'e}quation fonctionnelle aux valeurs moyennes.
\newblock {\em Glasnik Mat.-Fiz. Astronom. Dru{\v s}tvo Mat. Fiz. Hrvatske Ser.
  II}, 13:243--248, 1958.

\bibitem{BarPal2}
M.~Barczy and {\relax Zs}.~P{\'a}les.
\newblock Existence and uniqueness of weighted generalized $\psi$-estimators.
\newblock arXiv 2211.06026, 2022.

\bibitem{BarPal3}
M.~Barczy and {\relax Zs}.~P{\'a}les.
\newblock Comparison and equality of generalized $\psi$-estimators.
\newblock arXiv 2309.04773, 2023.

\bibitem{Dar71b}
Z.~Dar{\'o}czy.
\newblock A general inequality for means.
\newblock {\em Aequationes Math.}, 7(1):16--21, 1971.

\bibitem{Dar72b}
Z.~Dar{\'o}czy.
\newblock {\"U}ber eine {K}lasse von {M}ittelwerten.
\newblock {\em Publ. Math. Debrecen}, 19:211--217 (1973), 1972.

\bibitem{Gin1938}
C.~Gini.
\newblock Di una formula comprensiva delle medie.
\newblock {\em Metron}, 13(2):3--22, 1938.

\bibitem{GrePie}
H.~J. Greenberg and W.~P. Pierskalla.
\newblock A review of quasi-convex functions.
\newblock {\em Oper. Res.}, 19:1553--1570, 1971.

\bibitem{HarLitPol34}
G.~H. Hardy, J.~E. Littlewood, and G.~P{\'o}lya.
\newblock {\em Inequalities}.
\newblock Cambridge, at the University Press, 2nd edition, 1952.

\bibitem{Hub64}
P.~J. Huber.
\newblock Robust estimation of a location parameter.
\newblock {\em Ann. Math. Statist.}, 35:73--101, 1964.

\bibitem{Hub67}
P.~J. Huber.
\newblock The behavior of maximum likelihood estimates under nonstandard
  conditions.
\newblock In {\em Proc. {F}ifth {B}erkeley {S}ympos. {M}ath. {S}tatist. and
  {P}robability ({B}erkeley, {C}alif., 1965/66), {V}ol. {I}: {S}tatistics},
  pages 221--233. Univ. California Press, Berkeley, Calif., 1967.

\bibitem{Kos}
M.~R. Kosorok.
\newblock {\em Introduction to Empirical Processes and Semiparametric
  Inference}.
\newblock Springer Series in Statistics. Springer, New York, 2008.

\bibitem{Pal82a}
{\relax Zs}.~P{\'a}les.
\newblock Characterization of quasideviation means.
\newblock {\em Acta Math. Acad. Sci. Hungar.}, 40(3-4):243--260, 1982.

\bibitem{Pal1984}
{\relax Zs}.~P\'{a}les.
\newblock Inequalities for comparison of means.
\newblock In {\em General inequalities, 4 ({O}berwolfach, 1983)}, volume~71 of
  {\em Internat. Schriftenreihe Numer. Math.}, pages 59--73. Birkh\"{a}user,
  Basel, 1984.

\bibitem{Vaa}
A.~W. van~der Vaart.
\newblock {\em Asymptotic {S}tatistics}, volume~3 of {\em Cambridge Series in
  Statistical and Probabilistic Mathematics}.
\newblock Cambridge University Press, Cambridge, 1998.

\end{thebibliography}

\end{document}